\documentclass[a4paper, reqno]{amsart}
\usepackage[margin=1.5in]{geometry}
\allowdisplaybreaks

\title[{}]{ Center of affine $\sll_{2|1}$ at the critical level}

\author{Dražen Adamović}
\address[D.A.]{Faculty of Science, Department of Mathematics, University of Zagreb, Bijenicka 30, 10 000 Zagreb, Croatia}
\email{adamovic@math.hr}

\author{Shigenori Nakatsuka}
\address[S.N.]{Department Mathematik, FAU Erlangen–Nürnberg, Cauerstraße 11, 91058, Erlangen, Germany}
\email{shigenori.nakatsuka@fau.de}

\usepackage{array}
\usepackage{latexsym}
\usepackage{amsmath}
\usepackage{amsfonts}
\usepackage{amssymb,nccmath}
\usepackage{mathtools}
\usepackage{amsxtra}
\usepackage{mathrsfs}
\usepackage{eucal}
\usepackage[all]{xy}
\usepackage{color}
\usepackage{braket}
\usepackage{graphics, setspace}
\usepackage{etoolbox, textcomp}
\usepackage{comment}
\usepackage{changepage}
\usepackage{xcolor}
\definecolor{rouge}{rgb}{0.85,0.1,.4}
\definecolor{bleu}{rgb}{0.1,0.2,0.9}
\definecolor{violet}{rgb}{0.7,0,0.8}
\usepackage[colorlinks=true,linkcolor=bleu,urlcolor=violet,citecolor=rouge]{hyperref}
\usepackage{cleveref}
\usepackage{lscape}
\usepackage{caption}
\usepackage{subcaption}
\usepackage{enumitem}
\usepackage{graphicx}
\usepackage{arydshln}

\usepackage{tikz}			
\usetikzlibrary{matrix}
\usetikzlibrary{automata, arrows.meta, positioning,cd}
\usepackage{dynkin-diagrams}
\usetikzlibrary{cd}			

\tikzset{%
  symbol/.style={
    draw=none,
    every to/.append style={
      edge node={node [sloped, allow upside down, auto=false]{$#1$}}
    },
  },
}

\newtheorem{definition}{Definition}[section]
\newtheorem{proposition}[definition]{Proposition}
\newtheorem{theorem}[definition]{Theorem}
\newtheorem{corollary}[definition]{Corollary}
\newtheorem{lemma}[definition]{Lemma}
\newtheorem{remark}[definition]{Remark}
\newtheorem{conjecture}[definition]{Conjecture}

\newtheorem*{conj}{Conjecture}
\newtheorem*{thm}{Main Theorem}

\numberwithin{equation}{section}

\newcommand{\Z}{\mathbb{Z}}

\newcommand{\C}{\mathbb{C}}

\newcommand{\End}{\operatorname{End}}

\newcommand{\Span}{\operatorname{Span}}
\newcommand{\Com}{\operatorname{Com}}

\newcommand{\Ker}{\operatorname{Ker}}

\renewcommand{\Im}{\operatorname{Im}}

\newcommand{\str}{\operatorname{str}}

\newcommand{\pd}{\partial}
\newcommand{\gr}{\operatorname{gr}}
\newcommand{\ad}{\operatorname{ad}}
\newcommand{\wt}{\operatorname{wt}}

\newcommand{\wun}{\mathbf{1}}
\newcommand{\kk}{\kappa}
\newcommand{\aff}{\mathrm{aff}}
\renewcommand{\geq}{\geqslant}
\renewcommand{\leq}{\leqslant}
\newcommand{\hwt}[1]{\mathrm{e}^{{#1}}}

\newcommand{\g}{\mathfrak{g}}
\newcommand{\h}{\mathfrak{h}}

\newcommand{\gl}{\mathfrak{gl}}
\newcommand{\sll}{\mathfrak{sl}}

\newcommand{\kw}[1]{\mathbf{#1}}
\newcommand{\gh}{\varphi}
\newcommand{\lm}[2]{[#1 {}_\lambda #2]}
\newcommand{\cent}{\mathfrak{z}}
\newcommand{\nil}{\mathfrak{n}}

\newcommand{\poch}[1]{\left(#1; q \right)_\infty}

\newcommand{\bea}{\begin{eqnarray}}
\newcommand{\eea}{\end{eqnarray}}

\newcommand{\W}{\mathcal{W}}

\newcommand\doi[2]{\href{http://dx.doi.org/#1}{#2}}

\begin{document}
\maketitle

\begin{abstract} In this article, we shall describe the center of the universal affine vertex superalgebra $V^{\kk_c}(\g)$ associated with $\g=\sll_{2|1}, \gl_{2|1}$ at the critical level $\kappa_c$ and prove the conjecture of  A. Molev and E. Ragoucy  \cite{MR} in this case.

The center $\cent(V^{\kk_c}(\sll_{2|1}))$ turns out to be isomorphic to the large level limit $\ell \rightarrow \infty$ of a vertex subalgebra, called the parafermion vertex algebra $K^{\ell} (\sll_2)$, of the affine vertex algebra $V^\ell(\sll_2)$. 
The key ingredient of the proof is to understand the principal $\W$-superalgebra $\W^{\kk_c}(\sll_{2|1})$ at the critical level. It relates the center $\cent(V^{\kk_c}(\sll_{2|1}))$ to $V^\infty(\sll_2)$ via the Kazama--Suzuki duality while it has a surprising coincidence with $V^{\kk_c}(\gl_{1|1})$, whose center has been recently described. Moreover, the centers $\cent(V^{\kk_c}(\sll_{2|1}))$ and $\cent(\W^{\kk_c}(\sll_{2|1}))$ are proven to coincide as a byproduct. 

A  general conjecture is proposed which describes the center $\cent(V^{\kk_c}(\sll_{n|m}))$ with $n>m$ as a large level limit of ``the dual side'', i.e.,\ the parafermion-type subalgebras of $\W$-algebras $\W^\ell(\sll_{n}, \mathbb{O}_{[n-m,1^m]})$ associated with hook-type partitions $[n-m,1^m]$, known also as vertex algebras at the corner.
\end{abstract}

\section{Introduction}\label{Intro}

Let $V^k(\g)$ be the universal affine vertex superalgebra associated with a basic-classical simple Lie superalgebra $\g$ and a complex number $k$, called a level. 
The center $ \cent(V^{k }(\g))$ is known to be trivial except for one level, called the \emph{critical level} $k= \kk_c$.
When $\g$ is a simple Lie algebra, the center $\cent(V^{\kk_c}(\g))$ at the critical level was described by B. Feigin and E. Frenkel in \cite{FF1, FF2}. It is isomorphic to the principal $\W$-algebra $\W^{\kk_c}(\g)$ and thus the ring of functions over the space of certain connections over the formal disc with values in the Langlands dual Lie algebra by the Feigin--Frenkel duality. This provides a fundamental fact in the geometric Langlands program, see e.g.\ \cite{F}. 
Recently, new explicit formulas for the generators of $ \cent(V^{\kk_c }(\g))$ was given by A. Molev in \cite{M1}. 

However, $ \cent(V^{\kk_c }(\g))$ for Lie superalgebras remains mysterious for decades and is much more difficult to study since $ \cent(V^{\kk_c}(\g))$  does not need to be finitely generated.
A. Molev and   E. Ragoucy started in \cite{MR} a systematic study on generators of  $ \cent(V^{\kk_c }(\g))$ when $\g = \gl_{n \vert m}$.  They constructed a family of higher rank Segal--Sugawara vectors which belong to $\cent(V^{\kk_c }(\g))$ and conjectured that they generate the center (see also \S \ref{Segal-Sugawara-vectors} for  a review). Their conjecture  has only been  proved   in the case $\g = \gl_{1 \vert 1}$ by A. Molev and E. Mukhin \cite{MM}.

The alternative description of $ \cent(V^{\kk_c }( \gl_{1 \vert 1}))$ was obtained \cite{A} by realizing $V^{\kk_c }( \gl_{1 \vert 1})$ as an orbifold of the vertex algebra $M \otimes \mathcal A$, where $M$ is the commutative vertex algebra of differential polynomials in two variables and $\mathcal{A}$  is the fermionic vertex algebra (denoted by $F$ in \cite{A}). Using this explicit realization (which is parallel to the realization of 
$L_1 (\mathfrak{gl}_{ m\vert n})$ from \cite{KW01}), it was proved in \cite{A}  that $\cent(V^{\kk_c }( \gl_{1 \vert 1})) \cong M_0$, where $M_0$ is an orbifold of $M$  which has the same  Hilbert--Poincar\'{e} series as  of  the center  determined in \cite{MM} (see \S  \ref{sec: FF center for affine gl11}, for more details on this construction).

In this paper, we consider the case $\g = \sll_{2|1}, \gl_{2|1}$ and completely describe $ \cent(V^{\kk_c}(\g))$.  Our method is based on finding a relation between the centers of $V^{\kk_c}(\sll_{2|1})$ and $V^{\kk_c }( \gl_{1 \vert 1})$ through the center of the \emph{$\W$-superalgebra} $\W^{\kk}(\sll_{2|1})$, which is obtained as a quantum Hamiltonian reduction of $V^\kk(\sll_{2|1})$ \cite{KRW03}.  The main theorem of the paper is the following.
\begin{thm}[Theorem \ref{Detecting the FF center}/ Corollary \ref{Final description of FF center}]\hspace{0mm}\\
There are isomorphisms of (commutative) vertex algebras 
    \begin{align*}
        &\cent(V^{\kk_c}(\g))\simeq \cent(\W^{\kk_c}(\g)) \simeq K^\infty(\check{\g}).
    \end{align*}
    where we set $\check{\g}=\sll_{2}, \gl_{2}$, respectively.
\end{thm}
\noindent
The vertex algera $K^\infty(\check{\g})$ appearing in the right-most hand side is the large level limit ($\ell\rightarrow \infty$) of the \emph{parafermion vertex algebra} $K^\ell(\check{\g})$ (cf. \cite{DLY, DLYW}) obtained as the coset subalgebra inside $V^\ell(\check{\g})$ by the Heisenberg vertex subalgebra associated with the Cartan subalgebra of $\sll_2$.

The first isomorphism for $\g=\gl_{2|1}$ implies a conjecture by A. Molev and E. Ragoucy  \cite{MR} in this case, which describes a generating set of $\cent(V^{\kk_c}(\gl_{2|1}))$ and also a conjecture by A. Molev and E. Mukhin \cite{MM} in this case, which asserts its associated graded algebra is isomorphic to the algebra of affine supersymmetric polynomials assocaited with the superspace $\C^{2|1}$. On the other hand, the second isomorphism comes from a ``correct'' form of the Kazama--Suzuki duality \cite{KaSu}, which is known except for this level. It may be regarded as a variant of the Feigin--Frenkel duality \cite{FF1, FF2} for the principal $\W$-algebras, which plays a crutial role to understand the center $\cent(V^{\kk_c}(\g))$ for Lie algebras $\g$. 

Let us briefly describe the proof of Main Theorem. 
The key ingredient is a rather surprising isomorphism $\W^{\kk_c}(\sll_{2|1}) \simeq V^{\kk_{c}}(\gl_{1|1})$ available only at this level (Proposition \ref{Wakimoto realization for principal sl(2|1)}). 
This isomophism, in particular, identifies their centers. 
We utilize the Wakimoto-type free field realizations of $V^{\kk_c}(\g)$ and $\W^{\kk_c}(\g)$ (\S \ref{sec: Wakimoto realization}), which provides an upper bound of $\cent(V^{\kk_c}(\g))$ by $\cent(\W^{\kk_c}(\g))$, i.e., the known center $\cent(V^{\kk_{c}}(\gl_{1|1}))$ \cite{MM}. 
This upper bound will be shown to be the lower bound of $\cent(V^{\kk_c}(\g))$ by using higher-rank Segal--Sugawara vector and a structure of algebras of affine supersymmetric functions. The second isomorphism is derived from the Kazama--Suzuki duality together with the description of the center $\cent(V^{\kk_c}(\gl_{1|1}))$ in \cite{A}.

Although the first isomorphism will be obtained first by the map induced by the restriction of the quantum Hamiltonian reduction, we provide in \S \ref{iqhr} a yet another one in the spirit of \emph{inverse Hamiltonian reduction}, which might be useful to construct various $V^{\kk_c}(\g)$-modules with arbitrary central characters.

\vskip 5mm
As already mentioned above, the Kazama--Suzuki duality for $\W^\kk(\sll_{2|1})$ is a key ingredient to describe the center $\cent(V^{\kk_c}(\sll_{2|1}))$ and plays a role of the Feigin--Frenkel duality for principal $\W$-algebras \cite{FF1, FF2}.
Indeed, the Feigin--Frenkel duality has a variant for $\W^\kk(\sll_{n|1})$, called the Feigin--Semikhatov duality --- which includes the Kazama--Suzuki duality --- \cite{CGN, CL4, FS} and a more general conjecture has recently been proposed in \cite{CFLN} for $\W^\kk(\sll_{n|m})$ ($n>m$).
The latter relates $\W^\kk(\sll_{n|m})$ to the $\W$-algebra $\W^\ell(\sll_n, \mathbb{O}_{[n-m,1^m]})$ associated with the nilpontent orbit $\mathbb{O}_{[n-m,1^m]}$ corresponding to a hook-type partition $[n-m,1^m]$.
In this setting, the parafermion $K^\ell(\sll_2)$ in the case $n=m=1$ is generalized to the coset subalgebra $C^\ell(\sll_{n},\mathbb{O}_{[n-m,1^m]})$ of  $\W^\ell(\sll_n, \mathbb{O}_{[n-m,1^m]})$ with respect to the affine vertex subalgebra for $\gl_m$. We note that it is known also as the vertex algebras at the corner \cite{CL4, GR}.
Base on this, we propose the following conjecture.
\begin{conj}
For $\g=\sll_{n|m}, \gl_{n|m}$ with $n>m$, there are isomorphisms
\begin{align*}
\cent(V^{\kk_c}(\g))\simeq \cent(\W^{\kk_c}(\g)) \simeq C^\infty(\g,\mathbb{O}_{[n-m,1^m]})
\end{align*}
where we set $\check{\g}=\sll_{n}, \gl_{n}$, respectively.
\end{conj}
\noindent
The structure of $\W^{\kk}(\g)$ at the critical level $\kk=\kk_c$, espacially in the case $m=1$, seems to be simple enough to understand the algebraic structure among higher rank Segal--Sugawara vectors in $\cent(V^{\kk_c}(\g))$ in the super setting. 
We hope to come back to this point in our forthcoming paper.
\vspace{1em}

\paragraph{\textbf{Acknowledgements}}  
We would like to thank A. Molev and M. Gorelik for long time discussions about center at the critical level for Lie superalgebras.

 D.A. is partially supported by the Croatian Science Foundation under the project IP-2022-10-9006 and by the project “Implementation of cutting-edge research and its application as part of the Scientific Center of Excellence QuantiXLie“, PK.1.1.02, European Union, European Regional Development Fund.
 
S.N. is partially supported by JSPS Kakenhi Grant number 21H04993.
A part of this work was done while S.N stayed Univesity of Alberta and visited University of Zagreb and  the RIMS of Kyoto. S.N. thanks the hospitality of these institutes.

\section{Affine vertex superalgebras}
Let $\g=\gl_{m|n}$ be the general linear Lie superalgebra associated with the vector superspace $V=\C^{m|n}$ and $e_{i,j}$ ($i,j=1,\dots, m+n$) be the $(i,j)$-th elementary matrix. 
The even supersymmetric invariant bilinear forms on $\gl_{m|n}$ are given by
\begin{align*}
    \kk=k_1 \kk_{V}+ k_2 \kk_{\g}
\end{align*}
where $\kk_V$ and $\kk_\g$ are the supertrace on $V$ and the Killing form, that is,
\begin{align*}
    \kk_V(u,v)=\str_V(uv),\quad \kk_\g=\str_\g(\ad_u \ad_v ).
\end{align*}
The affine Lie algebra $\widehat{\g}$ is the central extension of the loop algebra $\g[t^{\pm1}]$
\begin{align*}
    0\rightarrow \C K \rightarrow \widehat{\g}\rightarrow \g[t^{\pm1}] \rightarrow 0
\end{align*}
with Lie bracket 
\begin{align*}
    [u_{(p)},v_{(q)}]:=[u,v]_{(p+q)}+\kk(u,v) p\delta_{p+q,0}K,\quad [K,\widehat{\g}]=0
\end{align*}
where $u_{(p)}=ut^p$.
The (universal) affine vertex superalgebra associated with $\g$ at level $\kk$ is defined as the parabolic induction 
\begin{align*}
    V^\kk(\g):=U(\widehat{\g})\otimes_{U(\g[t]\oplus \C K)} \C
\end{align*}
induced from the $(\g[t]\oplus \C K)$-module $\C$ where $\g[t]$ acts by $0$ and $K$ by the identity, equipped with a unique vertex algebra structure with the unit $\wun=1\otimes1$, the translation operator $\pd$ satisfying $[\pd,u_{(p)}]=-p u_{(p-1)}$, and the state-field correspondence $Y(\cdot,z)$ satisfying $Y(u,z)=\sum u_{(p)} z^{-p-1}$. 
Here $u$ is identified with $u_{(-1)}\otimes \wun$. 

Note that for $n \neq m$, the decomposition $\gl_{n|m}=\sll_{n|m}\oplus \C h$ with $h=\sum e_{i,i}$ induces an isomorphism of vertex superalgebras
\begin{align*}
    V^\kk(\gl_{m|n})\simeq  V^\kk(\sll_{m|n}) \otimes  \pi^{\kk}_{h}.
\end{align*}
Here $\pi^{\kk}_{h}$ is the Heisenberg vertex algebra $V^{\kk}(\gl_1)$ with $\gl_1=\C h$.
It is straightforward to check the following formula of the $\lambda$-bracket
\begin{align*}
    &\lm{u}{v}=[u,v]+k\kk_{V}(u,v),\quad \lm{h}{h}=(m-n)k_1\lambda
\end{align*}
for $u,v \in \sll_{m|n}$ where we set $k=k_1+2(m-n)k_2$.

The level $\kk=-\frac{1}{2}\kk_\g$, denoted by $\kk_c$ is called \emph{the critical level}; the bilinear form is explicitly given by  
\begin{align*}
    \kappa_c(e_{i,j},e_{p,q})=\delta_{i,j}\delta_{p,q}(-1)^{p(i)+p(q)}-(m-n)\delta_{i,j}\delta_{p,q}(-1)^{p(i)}
\end{align*}
where $p(i)$ is the parity, which is 0 (resp. 1) if $i\leq m$ (resp. $i>m$).

Given a vertex superalgebra $V$ and a subalgebra $W$, the coset $\Com(W,V)$ is the subalgebra defined by 
\begin{align*}
    \Com(W,V):=\{u \in V ; [Y(u,z),Y(v,w)]=0\ (\forall v\in W)\}.
\end{align*}
In particular, the coset subalgebra $\mathfrak{z}(V):=\Com(V,V)$
is called the center of $V$. 
For a simple Lie algebra $\g$, the center $\mathfrak{z}(V^\kk(\g))$ is known to be trivial $\C$ at $\kk\neq \kk_c$, whereas non-trivial at the critical level $\kk=\kk_c$.
It is known as the Feigin--Frenkel center and explicitly described by using the notion of \emph{opers} associated with the Langlands dual of $\g$ \cite{FF1}.
On the other hand, to the best of our knowledge, the center for the basic-classical simple Lie superalgebras, the Feigin-Frenkel center $ \cent(V^{\kk_c}(\g))$ is not understood except for the case $\g=\gl_{1|1}$ \cite{A, MM}. 
In this paper, we determine the Feigin-Frenkel center for $\g=\gl_{2|1}$ or equivalently $\sll_{2|1}$ at the critical level $\kk=\kk_c$.

\section{Feigin-Frenkel center of affine $\gl_{1|1}$}\label{sec: FF center for affine gl11}

The Feigin-Frenkel center $\cent(V^{\kk_c}(\g))$ for $\g=\gl_{1|1}$ is determined in \cite{A, MM}. 
We recall the presentation which will play a key role, following \cite{A}. 

Let $\mathcal{A}_{\phi}$ denote the $bc$-system vertex superalgebra generated by odd fields $\phi, \phi^*$ satisfying the $\lambda$-bracket
\begin{align*}
\lm{\phi}{\phi^*}=1,\quad \lm{\phi}{\phi}=\lm{\phi^*}{\phi^*}=0.
\end{align*}
Let $x := \xi_+ \xi_- $ where we omit the symbol of normally ordered product. Then the operator $x_{(0)}$ define on  $\mathcal{A}_{\phi}$ the following $\Z$--gradation
$$ \mathcal{A}_{\phi} = \bigoplus_{\ell \in {\Z}}   (\mathcal{A}_{\phi})_{\ell},  $$ 
where  $$(\mathcal{A}_{\phi})_{\ell} = \{v \in  \mathcal{A}_{\phi} \ \vert \ x_{(0)} v = \ell v\}. $$ 

Let
$$M=\C[\partial^n \xi_+, \partial^n \xi_- \mid n\geq 0]$$ 
denote the commutative vertex algebra of differential polynomials in variables $\xi_+,\xi_-$. 
We introduce the weight decomposition 
\begin{align*}
M=\bigoplus_{n\in \Z} M_n
\end{align*}
by setting $\wt(\xi_\pm)=\pm1$ with $\wt(\pd A)=\wt(A)$ and $\wt(AB)=\wt(A)+\wt(B)$.
The subspace $M_0\subset M$ is a vertex subalgebra admitting the following generating sets
\begin{align}\label{strong generators at wt=0}
M_0=\langle (\pd^n\xi_+)(\pd^m\xi_-) \mid n,m\geq0 \rangle = \langle (\pd^n\xi_+) \xi_- \mid n\geq0 \rangle
\end{align}
as a differential algebra.

By \cite[Theorem 3.2 (1)]{A}, there is an injective homomorphism of vertex superalgebras
\begin{align}\label{realization of affine gl11 at critical level}
\begin{array}{cccc}
q\colon & V^{\kk_{c}}(\gl_{1|1}) &\hookrightarrow & \mathcal{A}_{\phi}\otimes M\\
&e_{1,2},\ e_{2,1} &\mapsto &  \phi \xi_+,\ \phi^* \xi_- \\
&e_{1,1},\ e_{2,2} &\mapsto &  \phi \phi^*,\ \xi_+\xi_--\phi \phi^* .
\end{array}
\end{align}
Moreover, it was proved that
 the image of $q$ is identified with the following $U(1)$--orbifold of $ \mathcal{A}_{\phi}\otimes M$:
  $$  V^{\kk_c}(\gl_{1|1}) \xrightarrow{\simeq}  (\mathcal{A}_{\phi} \otimes M)^{U(1)} :=  \bigoplus _{\ell  \in {\Z} }  (\mathcal{A}_{\phi})_{\ell} \otimes M_{-\ell} $$
  which then easily  identifies  the center of $V^{\kk_c}(\gl_{1|1})$ with  the vertex algebra $M_0$.
\begin{theorem}[{\cite[Theorem 3.2]{A}}]\label{FF center for gl11}
The map $q$ induces an isomorphism of vertex algebras
\begin{align*}
\cent(V^{\kk_c}(\gl_{1|1}))\simeq M_0.
\end{align*} 
\end{theorem}


\section{Center of $\W^\kk(\sll_{2|1})$ at the critical level}
\subsection{$\W$-superalgebra}
To understand the Feigin--Frenkel center $\cent(V^{\kk_c}(\gl_{2|1}))$, we first consider the center of the $\W$-superalgebra $\W^\kk(\gl_{2|1})$ associated with the regular nilpotent orbit (the orbit containing $f=e_{2,1}$) at the critical level $\kk=\kk_c$ \cite{KRW03}. 

The vertex superalgebra $\W^\kk(\gl_{2|1})$ is, by definition, the vertex superalgebra obtained as the cohomology of the following BRST complex
\begin{align}\label{BRST complex}
C^\bullet_{DS}(V^\kk(\gl_{2|1})):=V^\kk(\gl_{2|1})\otimes \bigwedge{}^{\bullet}
\end{align}
equipped with the differential
\begin{align}\label{BRST differential}
d=\int Y(Q,z) \mathrm{d}z,\quad Q=(e_{1,2}+1)\gh_{2}^*- e_{1,3}\gh_{3}^*.
\end{align}
Here, $\bigwedge{}^{\bullet}$ is the charged fermion vertex superalgebra, strongly generated by $\varphi_{i}, \varphi_{i}^*$ ($i=2,3$) which have parity opposite to $e_{1,i}$ in $\gl_{2|1}$ and satisfy the $\lambda$-bracket
\begin{align*}
    \lm{\varphi_{i}}{\varphi_{j}^*}=\delta_{i,j},\quad \lm{\varphi_{i}}{\varphi_{j}}=\lm{\varphi_{i}^*}{\varphi_{j}^*}=0.
\end{align*}
By replacing $V^\kk(\gl_{2|1})$ with $V^\kk(\gl_{2|1})$-modules $M$ in \eqref{BRST complex}, we obtain the BRST complex $C^\bullet_{DS}(M)$, whose cohomology
$$H_{DS}^\bullet(M):=H^\bullet(C^\bullet_{DS}(M),d)$$
is naturally a module over the $\W$-superalgebra $\W^\kk(\gl_{2|1}):=H_{DS}^\bullet(V^\kk(\gl_{2|1}).$

By \cite{KW04}, the complex decomposes into a tensor product
\begin{align*}
C^\bullet_{DS}(V^\kk(\gl_{2|1}))\simeq C^\bullet_+\otimes C^\bullet_0
\end{align*}
where $C^\bullet_+$ is the subcomplex generated by $e_{1,2}, e_{1,3}, \varphi_{2}, \varphi_{3}$ and $C^\bullet_0$ generated by 
\begin{align*}
&\kw{e}_{1,1}=e_{1,1}+\gh_2\gh_2^*-\gh_3\gh_3^*,\\
&\kw{e}_{2,2}=e_{2,2}-\gh_2\gh_2^*,\quad \kw{e}_{3,3}=e_{3,3}+\gh_3\gh_3^*\\
&\kw{e}_{2,1}=e_{2,1},\quad \kw{e}_{2,3}=e_{2,3}+\gh_3\gh_2^*\\
&\kw{e}_{3,1}=e_{3,1},\quad \kw{e}_{3,2}=e_{3,2}+\gh_2\gh_3^*\\
& \varphi_{2}^*,\quad \varphi_{3}^*.
\end{align*}
We have the cohomology vanishing 
\begin{align*}
    H^{n}\left(C^\bullet_{DS}(V^\kk(\gl_{2|1}))\right)\simeq \delta_{n,0}\W^\kappa(\gl_{2|1}),\quad H^{n}\left(C^\bullet_+\right)\simeq \delta_{n,0} \C
\end{align*}
and then $\W^\kappa(\gl_{2|1})$ is realized as the kernel $\W^\kappa(\gl_{2|1})\simeq \Ker d \subset C_0^\bullet$.
Moreover, it admits a set of strong generators given by a basis of the centralizer
\begin{align*}
    (\gl_{2|1})^f=\Span\{h_0,\ h_1+2h_2,\ e_{2,1},\ e_{2,3},\ e_{3,1}\}
\end{align*}
where we set 
\begin{align*}
    h_0=e_{1,1}+e_{2,2}+e_{3,3},\quad h_1=e_{1,1}-e_{2,2},\quad h_2=e_{2,2}+e_{3,3}.
\end{align*}
Indeed, we find the following strong generators:
\begin{align}\label{KW generators}
\begin{split}
&h=\kw{h}_0,\quad J=-(\kw{h}_1+2\kw{h}_2),\\
&S=\kw{f}+(\kw{h}_1+\kw{h}_2)\kw{h}_2-\frac{k+1}{2} \pd \kw{h}_1-\pd \kw{h}_2+\kw{e}_{2,3}\kw{e}_{3,2},\\
&G^+=\kw{e}_{2,3}, \quad G^-=\kw{e}_{3,1}+ k \pd \kw{e}_{3,2}+(\kw{h}_1+\kw{h}_2)\kw{e}_{3,2}. 
\end{split}
\end{align}
Note that the decomposition 
$V^\kk(\gl_{2|1})\simeq  V^\kk(\sll_{2|1}) \otimes  \pi^{\kk}_{h}$ induces 
\begin{align*}
    \W^\kk(\gl_{2|1})\simeq
    \W^\kk(\sll_{2|1})\otimes \pi^{\kk}_{h}.
\end{align*}
Here, the $\W$-superalgebra $\W^\kk(\sll_{2|1})$ is defined by replacing $V^\kappa(\gl_{2|1})$ with $V^\kappa(\sll_{2|1})$ in \eqref{BRST complex} and explicitly generated by $J, S, G^\pm$ in $\W^\kk(\gl_{2|1})$.
It is straightforward to check that these generators satisfy the following $\lambda$-brackets:
\begin{align*}
&\lm{J}{J}=-(2k+1)\lambda,\quad \lm{J}{G^\pm}=\pm G^\pm,\\
&\lm{S}{S}=-(k+1)\left(\frac{3}{2}(k+1)(2k+1)\lambda^3+2S \lambda+\pd S \right),\\
&\lm{S}{J}=-(k+1)(J\lambda+\pd J),\quad \lm{S}{G^\pm}=-(k+1)\left(\frac{3}{2}G^\pm \lambda+ \pd G^\pm{}\right),\\
&\lm{G^+}{G^+}=\lm{G^-}{G^-}=0,\\
&\lm{G^+}{G^-}=(k+1)(2k+1)\lambda^2-(k+1)J\lambda+\left(S-\frac{k+1}{2}\pd J \right).
\end{align*}

At the critical level $\kk=\kk_c$ i.e. $k=-1$, we have
\begin{align*}
&\lm{J}{J}=\lambda,\quad \lm{J}{G^\pm}=\pm G^\pm,\quad  \lm{G^+}{G^-}=S
\end{align*}
with all the remaining $\lambda$-brackets to be zero.
Then, we observe the following.
\begin{proposition}\label{Wakimoto realization for principal sl(2|1)}
There is an isomorphism of vertex superalgebras:
\begin{align*}
\begin{array}{cccc}
\eta\colon &\W^{\kk_c}(\sll_{2|1}) &\xrightarrow{\simeq} &V^{\kk_{c}}(\gl_{1|1})\\
&G^+,G^- &\mapsto &e_{1,2}, e_{2,1}\\
&J,S &\mapsto &  e_{1,1}, e_{1,1}+e_{2,2}.
\end{array}
\end{align*}
In particular, there is an isomorphism of (commutative) vertex algebras
\begin{align*}
    \cent(\W^{\kk_c}(\gl_{2|1}))\simeq \pi^0_h \otimes \cent(V^{\kk_{c}}(\gl_{1|1})).
\end{align*}
\end{proposition}
Note that $\W^{\kk}(\sll_{2|1})$ has a conformal vector $L=-\frac{1}{k+1}S$ defined for non-critical levels $\kk\neq \kk_c$. 
It induces a quasi-conformal structure for all levels which assignes the gradings $1, 3/2,2$ to the generators $J, G^\pm,S$, respectively.
Let us denote by $\W_{\kk}(\sll_{2|1})$ the unique graded simple quotient of $\W^{\kk}(\sll_{2|1})$. 
\begin{corollary}\label{simple quotient of W}
    $\W_{\kk_c}(\sll_{2|1})  \simeq \mathcal \pi_{J}.$
\end{corollary}
\begin{proof}
In the graded simple quotient $\W_{\kk_c}(\sll_{2|1})$, $S$ is of degree 2 or $S=0$. Since $S$ is an central element, it must act as a scalar in a simple quotient, and thus $S=0$.
Then, $G^{\pm}$ are central odd elements in $\W_{\kk_c}(\sll_{2|1})$ and thus $G^{\pm}=0$ holds by the same consideration. Hence, $\W_{\kk_c}(\sll_{2|1})$ is generated by $J$.
\end{proof}

The next result shows that $\W^{\kk}(\sll_{2|1})$  has also a unique non-graded simple quotient on which $S$ acts by non-zero scalar.

\begin{corollary}\label{non-graded simple quotient of W} 
If the central element $S$ acts on a simple quotient of $\W^{\kk}(\sll_{2|1})$ by a non-zero scalar, then the quotient is isomorphic to the $bc$-system $\mathcal{A}_\phi$.
\end{corollary}
\begin{proof}
Let us denote such a quotient by $\W_{\kk_c}^{S\neq0}(\sll_{2|1})$.
By rescaling the generator $G^-$ if necessary, we can assume that $S$ acts as the identity without loss of generality.  Let $\widetilde W =   \W^{\kk_c}(\sll_{2|1}) / \langle S -\wun \rangle$.  Then $\widetilde W$ is a vertex superalgebra in which  $G^{\pm}$ generate the $bc$-system $\mathcal{A}_\phi$. 
As $H= G^+ G^-$ is a Heisenberg vector in  $\widetilde W$ such that
$$ [H _{\lambda} G^{\pm}] = \pm G^{\pm},  \ [H _{\lambda} H] = \lambda,  \   [H _{\lambda} J ]  = {\lambda},  $$
$J-H$ belongs to the center $\cent(\widetilde W)$. 
Hence, $\W_{\kk_c}^{S\neq0}(\sll_{2|1})$ is strognly generated by $G^{\pm}$ and thus isomorphic to $\mathcal{A}_\phi$. 
\end{proof}
 
\subsection{Kazama-Suzuki duality}
By combining the isomorphisms in Proposition \ref{Wakimoto realization for principal sl(2|1)} and Theorem \ref{FF center for gl11}, we obtain the following realization of the center
\begin{center}
\begin{tikzcd}[ column sep = huge]
			\W^{\kk_c}(\sll_{2|1})
			\arrow[r, hook]&
			\mathcal{A}_\phi\otimes M\\
		      \cent(\W^{\kk_c}(\sll_{2|1})) 
			\arrow[r, " \simeq "] \arrow[u,symbol=\subset]&
			M_0 \arrow[u,symbol=\subset].
\end{tikzcd}
\end{center}
This realization is recovered by taking the large level limit of the Kazama--Suzuki coset construction \cite{A-1999, CL3, KaSu} of $\W^\kk(\sll_{2|1})$ at non-critical levels $\kk\neq \kk_c$, known also as the $\mathcal{N}=2$ superconformal algebra.

Let $V^\ell(\sll_2)$ be the universal affine vertex algebra assocaited with $\sll_2$ at level $\ell\in \C$ generated by the $\sll_2$-triple $E,H,F$ and consider the diagonal Heisenberg vertex algebra $\pi^{H_\Delta} \subset\mathcal{A}_\phi\otimes V^\ell(\sll_2)$ generated by 
$H_\Delta=H-2\phi\phi^*$.
Then the Kazama--Suzuki coset construction asserts an isomorphism
\begin{align}\label{Kazama-Suzuki}
\W^\kk(\sll_{2|1}) \xrightarrow{\simeq} \Com\left(\pi^{H_\Delta},\mathcal{A}_\phi\otimes V^\ell(\sll_2)\right)
\end{align}
where the levels $k,\ell$ satisfies the relation
\begin{align}\label{duality level}
(k+1)(\ell+2)=1.
\end{align} 
The explicit embedding is given by
\renewcommand{\arraystretch}{1.3}
\begin{align*}
G^+, G^- &\mapsto  (k+1)\phi E,\ -(k+1)\phi^* F,\\
J & \mapsto  -(2k+1) \phi\phi^*+(k+1) H,\\
S & \mapsto  -(k+1)^2EF+\frac{1}{2}(k+1)^2 \pd H-(k+1)^2\phi \phi^*H\\
& \hspace{1cm}-\frac{1}{2}(k+1) (2k+1)(\phi^*\pd \phi +\phi \pd \phi^*{}).
\end{align*}
\renewcommand{\arraystretch}{1}

One may extend the map to the critical level $\kk=\kk_c$ (i.e.,\ $k=-1$) by considering the large level limit $\ell \rightarrow \infty$ as suggested by \eqref{duality level}. 
To make it precise, we start with the universal affine vertex algebra $V_{R^\circ}(\sll_2)$ over the ring $R^\circ=\C[(\sigma+2)^\pm]$ so that 
\begin{align*}
    V_{R^\circ}(\sll_2)\otimes_{R^\circ} \C_\ell \simeq V^\ell(\sll_2),\qquad (\ell\in \C \backslash \{-2\}) 
\end{align*}
where $\C_\ell$ is the 1-dimensioanal $R^\circ$-module such that $\sigma$ acts by $\ell$. 
Then we consider the subalgebra $V_R^\infty(\sll_2)$ generated by 
\begin{align*}
\frac{1}{\sigma+2}E,\quad \frac{1}{\sigma+2}H,\quad \frac{1}{\sigma+2}F
\end{align*}
over the ring $R=\C[(\sigma+2)^{-1}]$. The large level limit of $V^\ell(\sll_2)$ is by definition the specialization $V^\infty(\sll_2)=V_R^\infty(\sll_2)/((\sigma+2)^{-1})$.
It is an algebra of differential polynomials in the variables 
\begin{align*}
\overline{E}=\left[\frac{1}{\sigma+2}E\right],\quad \overline{H}=\left[\frac{1}{\sigma+2}H\right],\quad \overline{F}=\left[\frac{1}{\sigma+2}F\right]
\end{align*}
Then it is straightforward to show that \eqref{Kazama-Suzuki} induces the the Kazama--Suzuki coset construction at the critical level.
\begin{proposition}
There is an embedding of vertex superalgebras 
\begin{align*}
\begin{array}{ccc}
\W^{\kk_c}(\sll_{2|1}) &\hookrightarrow &\mathcal{A}_\phi\otimes V^\infty(\sll_2)\\
G^+,G^- &\mapsto &\phi \overline{E},\ -\phi^* \overline{F}\\
J,S &\mapsto &  \overline{H}+\phi\phi^*,\ -\overline{E}{} \overline{F}.
\end{array}
\end{align*}
\end{proposition}
Now, the above embedding recovers the embedding appearing in the beginning of this section through the following commutative diagram:
\begin{center}
\begin{tikzcd}[ column sep = huge]
			\W^{\kk_c}(\sll_{2|1})
			\arrow[d, "\eta"', " \simeq "]
			\arrow[r, hook]&
			\mathcal{A}_\phi\otimes V^\infty(\sll_2) \arrow[d, two heads, "\mathrm{id}\otimes \overline{\eta}"]
			\\
		      V^{\kk_c}(\gl_{1|1})
			\arrow[r, hook, "q"]&
			\mathcal{A}_\phi\otimes M
\end{tikzcd}
\end{center}
by setting a homomorphism of differential algebras 
\begin{align*}
\overline{\eta}\colon V^\infty(\sll_2) \rightarrow M,\quad \overline{E}\mapsto \xi_+,\ \overline{H}\mapsto 0,\ \overline{F}\mapsto -\xi_-.
\end{align*}
In this picture, the subalgebra $M_0$ is interpreted as $(V^\infty(\sll_2)/I_{V^\infty(\h)})^{\h}$ where $I_{V^\infty(\h)}$ is the ideal generated by $\pd^n \overline{H}$ ($n\geq0$) and the $\h$-action is given by the $\h$-weight.
By \cite{CL4}, R.H.S is identified with the large level limit of the parafermion vertex algebra $$K^\ell(\sll_2):=\Com(\pi^\ell_\h,V^\ell(\sll_2))$$
and the strong generators and their algebraic relations are studied using Weyl's invariant theory.
Similarly, let us introduce
$K^\ell(\gl_2):=\Com(\pi^\ell_{\h},V^\ell(\gl_2))$
where $\h$ is the Cartan subalgebra of $\sll_2\subset \gl_2$.
\begin{corollary}
There are isomorphisms of (commutative) vertex algebras 
\begin{align*}
&\cent(\W^{\kk_c}(\sll_{2|1}))\simeq K^\infty(\sll_2),\quad \cent(\W^{\kk_c}(\gl_{2|1}))\simeq K^\infty(\gl_2).
\end{align*}
\end{corollary}
\noindent

\section{Wakimoto realization}\label{sec: Wakimoto realization}
We introduce the Wakimoto realization for $V^{\kk}(\gl_{2|1})$ and $\W^{\kk}(\gl_{2|1})$ and then compare their centers at the critical level based on it.

Let $\pi^{\kk-\kk_c}_\h$ be the Heisenberg vertex algebras associated with $\h\subset \gl_{2|1}$ at level $\kk-\kk_c$. It is, by definition, generated by $\widehat{e}_1, \widehat{e}_2, \widehat{e}_3$ which satisfy the $\lambda$-bracket
\begin{align*}
    \lm{\widehat{e}_i}{\widehat{e}_j}=(\kappa-\kappa_c)(e_{i,i},e_{j,j})\lambda.
\end{align*}
For $\nu\in \h^*$, we denote by $\pi^{\kk-\kk_c}_{\h,\nu}$ the Fock module over $\pi^{\kk-\kk_c}$ of highest weight $\nu$. we write $\nu=\nu_1\alpha_1+\nu_2 \alpha_2$ by using the simple roots $\alpha_i$  whose root vectors are $e_{i,i+1}$. 

Let $\mathcal{A}$ be the vertex superalgebra generated by even fields $a,a^*$ and odd fields $\phi_i, \phi_i^*$ ($i=2,3$) which satisfy the $\lambda$-bracket
\begin{align*}
&\lm{a}{a^*}=1,\quad \lm{a}{a}=\lm{a^*}{a^*}=0,\quad\lm{\phi_i}{\phi_j^*}=\delta_{i,j}.
\end{align*}
The following can be checked by direct computation.
\begin{proposition}\label{Wakimoto for affine I}
There is an embedding of vertex superalgebras 
\begin{align*}
\rho \colon V^{\kk}(\gl_{2|1})\hookrightarrow \mathcal{A}\otimes \pi^{\kk-\kk_c}_\h
\end{align*}
which satisfies
\begin{align*}
&e_{1,2}\mapsto a,\quad e_{2,3}\mapsto \phi_2+a^*\phi_3,\quad e_{1,3}\mapsto \phi_3\\
&e_{1,1}\mapsto \widehat{e}_{1}-a^*a-\phi_3^*\phi_3,\quad 
e_{2,2}\mapsto \widehat{e}_{2}+a^*a-\phi_2^*\phi_2,\quad
e_{3,3}\mapsto \widehat{e}_{3}+\phi_2^*\phi_2+\phi_3^*\phi_3\\
&e_{2,1}\mapsto -a^*{}^2a-\phi_3^* \phi_2+a^* \phi_2^* \phi_2-a^* \phi_3^* \phi_3 + a^*(\widehat{e}_1-\widehat{e}_2)+ k \pd a^*\\
&e_{3,2}\mapsto \phi_3^*a+\phi_2^*(\widehat{e}_2+\widehat{e}_3)+(k+1)\pd \phi_2^*\\
&e_{3,1}\mapsto -\phi_3^*a^* a -\phi_3^* (\widehat{e}_1+\widehat{e}_3) +\phi_2^*\phi_3^* \phi_2 +a^*\phi_2^* (\widehat{e}_2+\widehat{e}_3)+(k+1) a^*+ \pd\phi_2^*-k \pd\phi_3^*.
\end{align*}
Moreover, for $k\neq -1$,  the image is contained in the kernel of the screening operators:
\begin{align*}
S_i=\int Y(e_{i,i+1}^R\mathrm{e}^{-\frac{1}{k+1}\alpha_i},z)\mathrm{d}z\colon \mathcal{A}\otimes \pi^{\kk-\kk_c}_\h\rightarrow \mathcal{A}\otimes \pi^{\kk-\kk_c}_{\h,-\alpha_i},\quad (i=1,2),
\end{align*}
with 
\begin{align}\label{coefficient of scr}
e_{1,2}^R=a+\phi_2^*\phi_3,\quad e_{2,3}^R=\phi_2.
\end{align}
\end{proposition}
Note that the association \eqref{coefficient of scr} extends to an embedding \begin{align*}
\rho^R\colon V^0(\nil_+) \hookrightarrow \mathcal{A}\otimes \pi^{\kk-\kk_c}_\h,\quad e_{1,2},e_{2,3},e_{1,3}\mapsto e_{1,2}^R, e_{2,3}^R, e_{1,3}^R
\end{align*}
with $e_{1,3}^R=\phi_3$. Here, $V^0(\nil_+)$ is the level 0 affine vertex superalgebra associated with the upper nilpotent subalgebra $\nil_+\subset \gl_{2|1}$. Note that the two embeddings
\begin{align}\label{left and right centralize}
\rho \colon V^0(\nil_+)\hookrightarrow \mathcal{A}\otimes \pi^{\kk-\kk_c}_\h \hookleftarrow V^0(\nil_+) \colon \rho^R
\end{align} 
centralize each other. 

\begin{lemma}\label{FF center goes Heis}
At the critical level $\kk=\kk_c$, we have 
\begin{align*}
\rho \colon \cent(V^{\kk_c}(\gl_{2|1})) \hookrightarrow \pi^{0}_\h \subset \mathcal{A}\otimes \pi^{0}_\h.
\end{align*}
\end{lemma}
\proof
One can show the assertion by adapting, to our setting, the proof of \cite[Lemma 7.1]{F} for the simple Lie algebras.
Observe that the lexicographically ordered monomials of the form
\begin{align*}
\prod_{p_j<0 }\widehat{e}_{i_j,p_j} \prod_{q_j<0 }e^R_{u_j,q_j}\prod_{r_j\leq 0 }\phi^*_{v_j,r_j} \prod_{s_j\leq 0 }a^*_{s_j} \mathbf{1} 
\end{align*}
form a basis of $\mathcal{A}\otimes \pi^{0}_\h$. 
Since 
$\cent(V^{\kk_c}(\gl_{2|1}))\subset \Com(V^0(\nil_+), V^{\kk_c}(\gl_{2|1}))$, the formula in Proposition \ref{Wakimoto for affine I} for $e_{1,2}, e_{2,3}, e_{1,3}$ implies that $\cent(V^{\kk_c}(\gl_{2|1}))$ embeds into the subspace spanned by 
\begin{align*}
\prod_{p_j<0 }\widehat{e}_{i_j,p_j} \prod_{q_j<0 }e^R_{u_j,p_j} \mathbf{1}.
\end{align*}
It follows from $\cent(V^{\kk_c}(\gl_{2|1}))\subset \Com(V^{\kk_c}(\h), V^{\kk_c}(\gl_{2|1}))$ that $\cent(V^{\kk_c}(\gl_{2|1}))$ lies in the subspace of $\h$-weight 0. 
By considering the $\h$-weight for $e^R_{1}$, $e^R_{2}$, $e^R_{3}$, we conclude that $\cent(V^{\kk_c}(\gl_{2|1}))$ lies in the subspace spanned by $\prod_{p_j<0 }\widehat{e}_{i_j,p_j} \mathbf{1}$, that is, $\pi^0_\h$. This completes the proof.
\endproof

Next, we apply the quantum Hamiltonian reduction to the Wakimoto realization in Proposition \ref{Wakimoto for affine I} and obtain a homomorphism 
\begin{align*}
[\rho]\colon \W^\kk(\gl_{2|1})\rightarrow H_{DS}^\bullet(\mathcal{A}\otimes \pi^{\kk-\kk_c}_\h).
\end{align*}
Let us introduce the subalgebras 
$$\overline{\mathcal{A}}=\langle \phi_2, \phi_2^*\rangle,\quad \mathcal{A}^+=\langle a,a^*, \phi_3, \phi_3^*\rangle.$$
\begin{lemma}
    There is an isomorphism of vertex superalgebras
\begin{align*}
H_{DS}^n(\mathcal{A}\otimes \pi^{\kk-\kk_c}_\h)\simeq 
\delta_{n,0} \left(\overline{\mathcal{A}}\otimes \pi^{\kk-\kk_c}_\h \right).
\end{align*} 
\end{lemma}
\proof
Since the differential $d$ in \eqref{BRST differential} is realized as 
\begin{align*}
d=\int Y((a+1)\gh_{2}^*- \phi_{3}\gh_{3}^* ,z) \mathrm{d}z
\end{align*}
by Proposition \ref{Wakimoto for affine I}, we have
\begin{align*}
H_{DS}^\bullet(\mathcal{A}\otimes \pi^{\kk-\kk_c}_\h)\simeq (\overline{\mathcal{A}}\otimes \pi^{\kk-\kk_c}_\h)\otimes H_{DS}^\bullet(\mathcal{A}^+).
\end{align*}
By using the (super) Poincar\'{e} lemma, it is straightforward to show 
$$H_{DS}^n(\mathcal{A}^+)\simeq \delta_{n,0}\C.$$
This completes the proof.
\endproof

Now, let us introduce the following elements 
\begin{align*}
    b_0=\widehat{e}_{1}+\widehat{e}_{2}+\widehat{e}_{3},\quad b_1=\widehat{e}_{1}-\widehat{e}_{2},\quad b_2=\widehat{e}_{2}+\widehat{e}_{3}.
\end{align*}
\begin{proposition}\label{Wakimoto for principal super I}
The homomorphism 
$$[\rho]\colon \W^\kk(\gl_{2|1})\rightarrow \overline{\mathcal{A}}\otimes \pi^{\kk-\kk_c}_\h$$
is an embedding and satisfies 
\begin{align}\label{cohomology class for principal super I}
\begin{split}
&h\mapsto b_0, \quad J=-(b_1+2b_2+x_2),\\
&S=(b_1+b_2)b_2-\frac{k+1}{2}\pd b_1-\pd b_2-\frac{k+1}{2}x_2^2-(k+1)\pd x_2,\\
&G^+=\phi_2,\\
&G^-=(b_1+b_2)b_2\phi_2^*+k \phi_2^* \pd b_2\\
&\hspace{1cm}+(k+1)(b_1+2b_2+x_2)\pd \phi_2^*+\frac{1}{2}(k+1)(2k+1)\pd ^2\phi_2^*
\end{split}
\end{align}
where $x_2=\phi_2\phi_2^*$.
Moreover, for $k\neq -1$,  the image is contained in the kernel of the following screening operators:
\begin{align*}
S_i=\int Y([e_i^R]\mathrm{e}^{-\frac{1}{k+1}\alpha_i},z)\mathrm{d}z\colon \overline{\mathcal{A}}\otimes \pi^{\kk-\kk_c}_\h\rightarrow \overline{\mathcal{A}}\otimes \pi^{\kk-\kk_c}_{\h,-\alpha_i}
\end{align*}
with 
\begin{align*}
[e_1^R]=\mathbf{1},\quad [e_2^R]=\phi_2.
\end{align*}
\end{proposition}
\proof
The proof is obtained straightforwardly by using Proposition \ref{Wakimoto for affine I} and the explicit presentation of the generators in \eqref{KW generators}.
\endproof
The following can be proven similarly to Lemma \ref{FF center goes Heis}.
\begin{lemma}\label{W-FF center goes Heis}
At the critical level $\kk=\kk_c$, we have 
\begin{align*}
[\rho] \colon \cent(\W^{\kk_c}(\gl_{2|1})) \hookrightarrow \pi^{0}_\h \subset \overline{\mathcal{A}}\otimes \pi^{0}_\h.
\end{align*}
\end{lemma}

Since the center
\begin{align*}
\cent(V^{\kk_c}(\gl_{2|1}))\subset V^{\kk_c}(\gl_{2|1}) \subset C_{DS}^\bullet(V^\kk(\gl_{2|1}))
\end{align*}
lies in the center of the BRST complex, we have, in particular, $\cent(V^{k_c}(\sll_{2|1}))\subset \mathrm{Ker}\ d$.
Thus, we have the natural homomorphism
\begin{align*}
\iota\colon \cent(V^{\kk_c}(\gl_{2|1})) \rightarrow \cent(\W^{\kk_c}(\gl_{2|1}))\subset \W^{\kk_c}(\gl_{2|1}). 
\end{align*}
\begin{lemma}\label{FF center lies in W-center}
The map $\iota$ is injective.
\end{lemma}
\proof
The assertion follows from the Wakimoto realizations for both hand-sides in Lemma \ref{FF center goes Heis}, \ref{W-FF center goes Heis} and Proposition \ref{Wakimoto for principal super I} thanks to the following commutative diagram:
    \begin{center}
		\begin{tikzcd}
			\cent(V^{\kk_c}(\gl_{2|1}))
			\arrow[d, hook, "\rho"']
			\arrow[r, "\iota"]&
			\cent(\W^{\kk_c}(\gl_{2|1})) \arrow[d, hook, "{[\rho]}"]
                \arrow[r,symbol=\subset]&
			\W^{\kk_c}(\gl_{2|1}) \arrow[d, hook, "{[\rho]}"]
			\\
		      \pi^0_{\h} \arrow[r, "\mathrm{id}"]&
			  \pi^0_{\h} \arrow[r,symbol=\subset]&
			\overline{\mathcal{A}}\otimes \pi^0_{\h}.
		\end{tikzcd}
	\end{center}
\endproof

\section{Feigin-Frenkel center of affine $\gl_{2|1}$}
\subsection{Segal--Sugawara vectors}
\label{Segal-Sugawara-vectors}
In \cite{MR}, a family of elements in $\cent(V^{\kk_c}(\gl_{m|n}))$ are constructed by exploiting a certain matrix-valued differential operator.
To construct them, let us introduce an element
\begin{align*}
    \pd+ \widehat{E}=\sum_{i,j}e_{i,j}\otimes (\delta_{i,j}\pd+ e_{i,j,-1})(-1)^{p(i)}
\end{align*}
in $R:=\End(\C^{m|n})\otimes U(\widehat{\gl}_{m|n}\rtimes \C\pd)$ where the second factor is the enveloping algebra of the Lie superalgebra $\widehat{\gl}_{m|n}\rtimes \C\pd$.
We use the Koszul sign rule for the product in $R$, namely, 
$$(a\otimes b)(a'\otimes b')=(-1)^{p(b) p(a')}aa'\otimes bb'$$
    for parity homogeneous elements $a,a',b,b'$, and denote by $\str$ the supertrace on the first factor:
\begin{align*}
    \str\colon R\rightarrow U(\widehat{\gl}_{m|n}\rtimes \C\pd).
\end{align*}
We compose it with the projection
 \begin{align*}
     U(\widehat{\gl}_{m|n}\rtimes \C\pd)\simeq U(\widehat{\gl}_{m|n})\otimes \C[\pd]\twoheadrightarrow V^\kk(\gl_{m|n})\otimes \C[\pd]
 \end{align*}
 and abuse the notation 
 \begin{align*}
    \str\colon R\rightarrow U(\widehat{\gl}_{m|n}\rtimes \C\pd)\twoheadrightarrow V^\kk(\gl_{m|n})\otimes \C[\pd].
\end{align*}
\begin{theorem}[\cite{MR}]\label{Segal-Sugawara vectors}
At the critical level $\kk=\kk_c$, the coefficients $s_{p,q}$ in the expansions 
\begin{align*}
    \str(\pd+\widehat{E})^p=s_{p,0}\pd^n+ s_{p,1}\pd^{p-1}+\dots+s_{p,p}
\end{align*}
for $p\geq0$ lie in the center $\cent(V^{\kk_c}(\gl_{m|n}))$.
\end{theorem}

\begin{conjecture}[\cite{MR}] \label{mol-rag} 
The family $\{ \partial ^r s_{p,p} \}$, for $r \ge 0$, $p >0$ generate  $\cent(V^{\kk_c}(\gl_{m|n}))$.
\end{conjecture}
 

\subsection{Affine supersymmetric polynomials}
The center $\cent(V^{\kk_c}(\gl_{m|n}))$ is closely related to the affine supersymmetric polynomials as observed in \cite{MM} for $\widehat{\gl}_{1|1}$.
We view this relationship for our case $\widehat{\gl}_{2|1}$ through the Wakimoto realization in \S \ref{sec: Wakimoto realization}. 

Introduce the polynomial ring $R^{m|n}$ by
$$R^{m|n}:=\C[u_1,\dots,u_m,v_1,\dots, v_n],$$
which admits the natural action of the symmetric group $\mathfrak{S}_{m+n}$.
The ring of supersymmetric polynomilas $\Lambda^{m|n}$ is the subalgebra 
\begin{align*}
\Lambda^{m|n}:=\left\{f ; f|_{u_m=t,v_n=-t}\in R^{m-1|n-1} \right\}\subset (R^{m|n})^{\mathfrak{S}_m\times \mathfrak{S}_n} \subset R^{m|n}.
\end{align*}
It is known \cite{Ma} that $\Lambda^{m|n}$ is generated by the power sums 
\begin{align*}
    s_p=(u_1^p+\cdots +u_m^p)-(-1)^p(v_1^p+\cdots +v_n^p)\quad (p>0).
\end{align*}
The ring of affine supersymmetric polynomials $\Lambda^{m|n}_{\aff}$ is the (commutative) vertex algebraic extension of $\Lambda^{m|n}$; we consider the commutative vertex algebra
$$R^{m|n}_{\infty}:=\C[\pd^N u_1,\dots,\pd^N u_m, \pd^Nv_1,\dots, \pd^Nv_n\mid N\geq0]$$
and introduce $\Lambda^{m|n}_{\aff}$ as the subalgebra generated by $\Lambda^{m|n}$ as a differential algebra:
$$\Lambda^{m|n}_{\aff}:=\langle \pd^N f\mid f\in \Lambda^{m|n}, N\geq 0 \rangle \subset R^{m|n}_{\infty}.$$

\begin{remark}\label{MM conjecture} 
\textup{The Conjecture \ref{mol-rag} is reformulated in \cite[Conjecture 4.3]{MM} as
$$ \Lambda^{m|n}_{\aff} \cong \gr\cent(V^{\kk_c}(\gl_{m|n})). $$
This conjecture was  proved in \cite{MM} for $m=n=1$. We shall prove the conjecture for $m=2$, $n=1$.}
\end{remark}

Now, we relate the affine supersymmetric polynomials to the Segal--Sugawara vectors (Theorem \ref{Segal-Sugawara vectors}) by applying the Wakimoto realization (Proposition \ref{Wakimoto for affine I}). Recall that given a vertex superalgebra $V$, Li filtration $F^\bullet V$ is a decreasing filtration defined by
\begin{align*}
    F^p V=\left\{a^1_{(-n_1-1)}\cdots a^r_{(-n_r-1)}b; \begin{array}{l}
      a^1,\dots,a^r, b \in V  \\
      n_i \geq0,\ n_1+\cdots +n_r \geq p
    \end{array} \right\}.
\end{align*}
The associated graded vector (super)space $\gr V=\oplus_p F^p V/F^{p+1}V$ has a structure of Poisson vertex (super)algebra, see \cite{Ar12, Li}. 
Moreover, given a homomorphism $f\colon V\rightarrow W$ of vertex superalgebras, the above association induces a homomorphism $\overline{f}\colon \gr V\rightarrow \gr W$ of Poisson vertex superalgebras.
Let us call this construction the semi-classical limit of $f\colon V\rightarrow W$. We apply this limit to the Wakimoto realization (Proposition \ref{Wakimoto for affine I}) and obtain the homomorphism
\begin{align*}
    \overline{\rho}\colon \gr V^\kk(\gl_{2|1})\rightarrow \gr (\mathcal{A}\otimes \pi^{\kk-\kk_c}_\h)\simeq  \gr \mathcal{A}\otimes \gr \pi^{\kk-\kk_c}_\h.
\end{align*}

\begin{proposition}\label{affine ss polys in FF center}
There is an embedding
$$\Lambda^{2|1}_{\aff}\subset \gr\cent(V^{\kk_c}(\gl_{2|1})).$$
\end{proposition}
\proof
Note that one has an isomorphism $R_\infty^{2|1} \xrightarrow{\simeq} \gr \pi^{0}_\h$ which maps
\begin{align*}
     u_1, u_2 \mapsto e_{1,1}, e_{2,2},\quad v_1 \mapsto  e_{3,3}.
\end{align*}
Hence, it suffices to show that $\Im \overline{\rho}$ contains the power sums $s_{p}$ ($p\geq0$) as $\overline{\rho}$ is an embedding and restricts to 
\begin{align*}
    \overline{\rho}\colon \gr\cent(V^{\kk_c}(\gl_{2|1})) \hookrightarrow \gr \pi^{0}_\h
\end{align*}
by Lemma \ref{FF center goes Heis}. 
By the formula in Proposition \ref{Wakimoto for affine I}, the image of the Segal--Sugawara vector $s_{p,p}$ agrees with 
\begin{align*}
    \str \left(\begin{array}{ccc}
        e_{1,1} &  &  \\
         & e_{2,2} &\\
         & & -e_{3,3}
    \end{array} \right)^p 
    &= e_{1,1}^p+e_{2,2}^p-(-e_{3,3})^p\\
    &= u_1^p+u_2^p-(-1)^pv_1^p=s_p.
\end{align*}
This completes the proof.
\endproof

\subsection{Center}
Now, we are ready to establish the main result of the paper.
\begin{theorem}\label{Detecting the FF center}\hspace{0mm}\\
\textup{(1)} The natural homomorphism $\iota$ is an isomorphism
\begin{align*}
\iota\colon \cent(V^{\kk_c}(\gl_{2|1}))\xrightarrow{\simeq} \cent(\W^{\kk_c}(\gl_{2|1})).
\end{align*}
\textup{(2)} There is an isomorphism of differential algebras 
\begin{align*}
 \gr\cent(V^{\kk_c}(\gl_{2|1}))\simeq \Lambda^{2|1}_{\aff}.
\end{align*}
\end{theorem}
\proof
By Lemma \ref{W-FF center goes Heis} and \ref{FF center lies in W-center}, we have embeddings
\begin{align*}
    \cent(V^\kk(\gl_{2|1}))\subset \cent(\W^\kk(\gl_{2|1})) \subset \pi^0_\h.
\end{align*}
On the other hand, we have $\Lambda^{2|1}_{\aff}\subset \gr\cent(V^{\kk_c}(\gl_{2|1}))$ by Proposition \ref{affine ss polys in FF center}.
Hence, it suffices to show 
\begin{align}\label{opposite inclusion}
    \gr\cent(\W^{\kk_c}(\gl_{2|1}))\subset \Lambda^{2|1}_{\aff}.
\end{align}
For this purpose, we take the semi-classical limit of the Wakimoto realization of $\W^{\kk_c}(\gl_{2|1})$ (Proposition \ref{Wakimoto for principal super I}) and obtain the embedding $\gr \W^{\kk_c}(\gl_{2|1})\subset \gr \overline{\mathcal{A}} \otimes \gr \pi^0_\h$, which restricts to 
\begin{align*}
    \gr \cent(\W^{\kk_c}(\gl_{2|1}))\subset \gr\pi^0_{\h}\simeq R^{2|1}_\infty
\end{align*}
by Lemma \ref{W-FF center goes Heis}. 
On the other hand, we have the isomorphism 
$$\gr \W^{\kk_c}(\gl_{2|1})\simeq \gr \pi_h \otimes \gr V^{\kk_c}(\gl_{1|1}) $$
by Proposition \ref{Wakimoto realization for principal sl(2|1)}. 
By \cite{MM}, there is an isomorphism $\gr \cent(V^{\kk_c}(\gl_{1|1}))\simeq \Lambda^{1|1}_\aff$
and $\gr \cent(V^{\kk_c}(\gl_{1|1}))$ is generated by the elements
$$\overline{h}_{p,p}=e_{1,1}^{p-1}(e_{1,1}+e_{2,2})+(p-1)e_{1,1}^{p-2}e_{2,1}e_{1,2}$$
for $p\geq 1$.
Therefore, $\gr \cent( \W^{\kk_c}(\gl_{2|1}))$ is generated by $h=u_1+u_2+v_1$ and the elements corresponding to $\overline{h}_{p,p}$ ($p\geq 1$), that is, 
\begin{align*}
    (-1)^{p-1}(u_1+u_2+2v_1)^{p-1}(u_1+v_1)(u_2+v_1).
\end{align*}
Hence, we have $\gr\cent(\W^{\kk_c}(\gl_{2|1})) \subset \Lambda^{2|1}_{\aff}$.
This completes the proof.
\endproof
\begin{remark} 
\textup{
One may indeed deduce Conjecture \ref{mol-rag} in our case by using a standard argument on positively-filtered algebras from Theorem \ref{Detecting the FF center} (2) and the fact that $\Lambda^{2|1}_{\aff}$ is generated as a differential algebra by the image of $s_{p,p}$'s under the associated graded.}
\end{remark}
\begin{corollary}\label{Final description of FF center}
    There are isomorphisms of (commutative) vertex algebras 
    \begin{align*}
        &\cent(V^{\kk_c}(\sll_{2|1}))\simeq K^\infty(\sll_2),\quad \cent(V^{\kk_c}(\gl_{2|1}))\simeq K^\infty(\gl_2).
    \end{align*}
\end{corollary}

\section{Inverse Hamiltonian reduction}
\label{iqhr}
In this section, we relate the centers of $V^{\kk_c}(\sll_{2|1})$ and $\W^{\kk_c}(\sll_{2|1})$ by realizing $V^{\kk_c}(\sll_{2|1})$ in terms of $\W^{\kk_c}(\sll_{2|1})$, namely the inverse Hamiltonian reduction first considered by Semikhatov \cite{Sem96} for non-critical levels.
For this realization, we use the $bc$-system vertex superalgebra $\mathcal{A}_\phi$ (see §\ref {sec: FF center for affine gl11}) and the so-called half-lattice vertex algebra $\Pi(0)$ introduced in \cite{A-2017}.
The algebra $\Pi(0)$ is, by definition, the subalgebra
$$\Pi(0)=\bigoplus_{n\in \Z}\pi^{u,v}_{n(u+v)}\subset V_L$$
inside the lattice vertex superalgebra $V_{L}$ associated with the integer lattice $L=\Z u\oplus \Z v$ equipped with bilinear form such that $(u,u)=1=-(v,v)$. It will be convenient to introduce 
$$c=u+v,\quad d=u-v.$$
As the highest weights in $\Pi(0)$ are $n(u+v)=nc$ which have norm zero, we may extend $\Pi(0)$ to 
$$\Pi^{\frac{1}{m}}(0)=\bigoplus_{n\in \Z}\pi^{u,v}_{\frac{n}{m}(u+v)}$$
along a larger lattice $\frac{1}{m}\Z c\supset \Z c$ for $m\in \Z_+$. By setting
\begin{align*}
    \mu=\frac{k}{2}c+d,\quad \nu=\frac{1}{2}\left(\frac{k}{2}c-d\right),
\end{align*}
we have the following realization of $V^\kk(\sll_{2|1})$.
\begin{proposition} \hspace{0mm}\\
\textup{(1) (cf.\ \cite{Sem96})} There is an embedding of vertex superalgebras 
$$\Phi^\kk\colon V^{\kk} (\sll_{2|1}) \hookrightarrow  \W^\kk(\sll_{2|1})\otimes \mathcal{A}_\phi \otimes \Pi^\frac{1}{2}(0)$$
such that
\begin{align*}
& h_1\mapsto \mu,\quad h_2 \mapsto \tfrac{1}{2}(J-\phi \phi^*-\mu),\\
&e_{1,2}\mapsto \hwt{c},\quad e_{1,3}\mapsto \phi \hwt{c/2},\quad e_{3,2}\mapsto \phi^* \hwt{c/2},\\
&e_{2,1}\mapsto (\mathcal{T}-\nu^2+(k+1)\pd \nu)\hwt{-c},\\
&e_{3,1}\mapsto -\left(G^++\phi^*\nu-(k+\tfrac{1}{2})\pd\phi^*-\tfrac{1}{2}J\phi^* \right)\hwt{-c/2},\\
&e_{2,3}\mapsto +\left(G^-+\phi\nu-(k+\tfrac{1}{2})\pd\phi+\tfrac{1}{2}J\phi \right)\hwt{-c/2},
\end{align*}
where 
$$\mathcal{T}=-S+G^+\phi+G^-\phi^*+\tfrac{1}{4}(J+\phi\phi^*)^2-\tfrac{1}{2}(k+1)(\phi\phi^*)^2.$$

\textup{(2)} At the critical level $\kk=\kk_c$, the map $\Phi^{\kk_c}$ restricts to an isomorphism
\begin{align*}
    \Phi^{\kk_c}\colon \cent(V^{\kk_c} (\sll_{2|1})) \xrightarrow{\simeq} \cent(\W^{\kk_c}(\sll_{2|1})).
\end{align*}

\end{proposition}

\proof
(1) The map $\Phi^\kappa$ essentially agrees with the one in \cite{Sem96} at non-critical levels $\kappa \neq \kappa_c$. We reformulate it regarding our choice of strong generators, which works for all levels $\kappa$. (One may check that the formula $\Phi^\kappa$ preserves the OPEs and thus defines a homomorphism of vertex superalgebras directly.)
To show that $\Phi^\kappa$ is an embedding, we consider the lexicographic order of the monomial basis for $\W^\kk(\sll_{2|1})\otimes \mathcal{A}_\phi \otimes \Pi(0)^{1/2}$ by introducing the following order for the generators
$$\pd^{\bullet}S>\pd^{\bullet}J>\pd^{\bullet}G^+>\pd^{\bullet}G^->\pd^{\bullet}\phi>\pd^{\bullet}\phi^*>\pd^{\bullet}c>\pd^{\bullet}d>\hwt{pc/2} $$
with $\pd^{n+1}A>\pd^{n}A$ and $\hwt{(p+1)c/2}>\hwt{pc/2}$.
Then the leading terms of the images of the strong generators of $V^\kappa(\sll_{2|1})$ are
\begin{align*}
& h_1\sim \mu,\quad h_2 \sim \tfrac{1}{2}J,\\
&e_{1,2}\sim \hwt{c},\quad e_{1,3}\sim \phi \hwt{c/2},\quad e_{3,2}\sim \phi^* \hwt{c/2},\\
&e_{2,1}\sim -S\hwt{-c}, \quad e_{3,1}\sim -G^+\hwt{-c/2},\quad e_{2,3}\sim +G^-\hwt{-c/2}.
\end{align*}
Now, it is clear that $\Phi^\kappa$ is injective. 

(2) As $\Phi^{\kk_c}$ is an embedding, it suffices to show 
\begin{align*} 
    \Phi^{\kk_c}\left(\cent(V^{\kk_c} (\sll_{2|1})) \right)&\subset  \cent(\W^{\kk_c}(\sll_{2|1}))
\end{align*}
thanks to Theorem \ref{Detecting the FF center} (1).
This is achieved by using a standard technique developed to show the almost simplicity of inverse Hamiltonian reductions (see e.g.\ the proof of \cite[Theorem 8.1]{A-2017}). 
Set $\mathfrak{Z}=\Phi^{\kk_c}\left(\cent(V^{\kk_c} (\sll_{2|1})) \right)$. 
Since $\mathfrak{Z}$ commutes with $\Phi^{\kk_c}(h_1)$ and $\Phi^{\kk_c}(\hwt{c})$, we have 
$\mathfrak{Z}\subset \W^\kk(\sll_{2|1})\otimes \mathcal{A}_\phi$.
Secondly, $\mathfrak{Z}$ commutes with $\Phi^{\kk_c}(e_{1,3})$ and $\Phi^{\kk_c}(e_{3,2})$. 
Then it follows that 
\begin{align*}
    \mathfrak{Z}\subset \W^\kk(\sll_{2|1}) \subset \W^\kk(\sll_{2|1})\otimes \mathcal{A}_\phi.
\end{align*}
Finally, $\mathfrak{Z}$ commutes with $\Phi^{\kk_c}(h_2)$, $\Phi^{\kk_c}(e_{3,1})$ and $\Phi^{\kk_c}(e_{2,3})$ implies that $\mathfrak{Z}\subset \W^\kk(\sll_{2|1})$ commutes with $J$, $G^\pm$. As $S$ is a central element, we conclude $\mathfrak{Z}\subset \cent(\W^\kk(\sll_{2|1}))$. This completes the proof.
\endproof

Moreover, the map $\Phi^\kk$ at the critical level relates the maximal ideals.  We have the following result which gives another formulation of the  realization of $L_{\kk_c}(\sll_{2|1})$ obtained in \cite{KW01}.

\begin{proposition}
At the critical level $\kk=\kk_c$, the map $\Phi^{\kappa_c}$ induces an embedding for the (unique) graded simple quotients:
$$\Phi_{\kk_c}\colon L_{\kk_c} (\sll_{2|1}) \hookrightarrow  \W_{\kk_c}(\sll_{2|1})\otimes \mathcal{A}_\phi \otimes \Pi^\frac{1}{2}(0). $$
\end{proposition}
\proof
Let $\mathcal{A}_a$ denote the $\beta\gamma$-system vertex algebra, strongly generated by $a,a^*$ which satisfy the $\lambda$-brackets 
$$\lm{a}{a^*}=1,\quad \lm{a}{a}=\lm{a^*}{a^*}=0.$$
By \cite{KW01}, there is an embedding
\begin{align*}
   g\colon L_{\kk_c}(\sll_{2|1})\hookrightarrow \mathcal{A}_a^{\otimes2}\otimes \mathcal{A}_\phi.
\end{align*}
such that 
\begin{align*}
& h_1\mapsto -a_1^*a_1-a_2^*a_2,\quad h_2 \mapsto a_2^*a_2+\phi\phi^*,\\
&e_{1,2}\mapsto a_1a_2,\quad e_{1,3}\mapsto a_1\phi^*,\quad e_{3,2}\mapsto a_2\phi,\\
&e_{2,1}\mapsto -a^*_1a^*_2,\quad e_{3,1}\mapsto -a^*_1\phi,\quad e_{2,3}\mapsto a^*_2\phi^*.
\end{align*}
By combining it with the Friedan--Martinec--Shenker bosonization 
\begin{align*}
   \mathcal{A}_a \hookrightarrow \Pi(0),\quad a\mapsto \hwt{c},\ a^* \mapsto -u \hwt{-c},
\end{align*}
we obtain the following commutative diagram:
   \begin{center}
		\begin{tikzcd}
			V^{\kk_c}(\sll_{2|1})
			\arrow[d, two heads]
			\arrow[r, "\Phi^{\kk_c}"]&
			\W^{\kk_c}(\sll_{2|1})\otimes \mathcal{A}_\phi \otimes \Pi^\frac{1}{2}(0) \arrow[r, two heads] & \pi_J \otimes \mathcal{A}_\phi \otimes \Pi^\frac{1}{2}(0)
                \arrow[d,hook, "\Psi"]
			\\
		     L^{\kk_c}(\sll_{2|1}) \arrow[r, "g"]&
			  \mathcal{A}_a^{\otimes2}\otimes \mathcal{A}_\phi \arrow[r,symbol=\subset]&
			\Pi^{\frac{1}{2}}(0)^{\otimes2}\otimes \mathcal{A}_\phi.
		\end{tikzcd}
	\end{center}
Here we have used $\W_{\kk_c}(\sll_{2|1})\simeq \pi_J$ (Corollary \ref{simple quotient of W}) and $\Psi$ is given by
 \begin{align*}
 \begin{array}{lll}
     J\mapsto \tfrac{1}{2}(d_1-d_2)+x,& \phi\mapsto \hwt{\frac{1}{2}(c_1-c_2)}\phi^*,& \phi^*\mapsto \hwt{-\frac{1}{2}(c_1-c_2)}\phi,\\
     c\mapsto c_1+c_2,& \hwt{\frac{1}{2}pc}\mapsto \hwt{\frac{1}{2}p(c_1+c_2)},& d\mapsto \tfrac{1}{2}(d_1+d_2).
    \end{array}
 \end{align*}
Hence, $\Phi^{\kk_c}$ factors through the simple quotient as desired.
\endproof

\end{document}